\documentclass[10 pt]{amsart}
\usepackage[active]{srcltx}
\usepackage[mathscr]{eucal}
\usepackage{amsmath,amssymb,hyperref}
\usepackage{graphics}
\usepackage{enumerate}

\newtheorem{thm}{Théorème}[section]

\newtheorem{cor}[thm]{Corollaire}
\newtheorem{prop}[thm]{Proposition} 
\newtheorem{lemme}[thm]{Lemme}
\newtheorem{definition}[thm]{D\'efinition}
\newtheorem{rem}[thm]{Remarque}

\def\C{\mathbb C}

\def\R{\mathbb R}

\def\F{\mathcal F}
\input xy
\xyoption{all}
\usepackage[french]{babel}
\usepackage[T1]{fontenc}

\usepackage{amsfonts,eufrak}

\usepackage{amssymb, amsmath}

\DeclareMathOperator{\ddbar}{\partial\overline{\partial}}
\newcommand{\field}[1]{\mathbb{#1}}

\newcommand{\RR}{\field{R}}

\newcommand{\CF}{{\mathcal C}_{\F}^+}
\newcommand{\CFN}{{{\mathcal C}_{\F}^+}(1)}
\newcommand{\CFR}{{{\mathcal C}_{\F}^+}_{Res}}
\newcommand{\con}{N_\F^*}
\newcommand{\norm}{N_\F}

\begin{document}
 
\title[Feuilletages holomorphes]{ Feuilletages holomorphes admettant une mesure  transverse invariante}
\vskip 10 pt





\selectlanguage{french}

\author[FR\'ED\'ERIC TOUZET]{FR\'ED\'ERIC TOUZET$^*$}
\thanks{$^*$ IRMAR, Campus de Beaulieu 35042 Rennes Cedex, France, frederic.touzet@univ-rennes1.fr}

\begin{abstract} 
Soit $\mathcal F$ un feuilletage holomorphe régulier de codimension 1 sur une variété kählerienne compacte. On suppose que $\mathcal F$ admet un courant positif invariant par holonomie.
Le but de cette note est d'établir l'alternative suivante:

- Il existe une hypersurface  invariante par le feuilletage.

- Le feuilletage admet une métrique hermitienne transverse de courbure constante invariante par holonomie.
\\
 
\noindent {\sc Abstract.}  Let $\F$ be a regular codimension 1 holomorphic foliation on a compact K\" ahler manifold. One assumes in addition that $\F$ possesses a transverse invariant positive current. The aim  of this paper is to establish the following alternative:

- There exists an invariant hypersurface.

- The foliation admits a transverse invariant hermitian metric with constant curvature.

\end{abstract}

\maketitle

\section{Introduction}

Soit $\F$ un feuilletage régulier sur une variété lisse. Rappelons qu'une mesure transverse invariante par ce feuilletage est définie par la donnée d'une mesure sur les transversales invariante par le pseudo-groupe d'holonomie de $\F$.\footnote{Toutes les mesures considérées ici sont boréliennes, positives et de masse localement finie.} Du point de la vue de la régularité, on peut considérer deux situations extrêmes: celle où la mesure est atomique (en l'occurence ici concentrée sur une feuille) et celle où elle est donnée comme forme volume d'une métrique riemannienne transverse invariante. 

Même quand la variété ambiante est compacte, l'existence de feuilletages mesurés et {\it a fortiori} de feuilletages transversalement riemanniens reste un phénomène rare. Dans le cas très particulier décrit par le théorème qui suit, nous nous proposons  de montrer que ces deux aspects, riemanniens et mesurés, sont essentiellement associés à la même classe de feuilletage pour peu que le support de la mesure soit suffisamment "épais".

\vskip 10 pt

\textbf{Théorème Principal}

{\it Soit $\F$ un feuilletage holomorphe de codimension un sur une variété Kähler compacte $X$; supposons que $\F$ admette une mesure diffuse (i.e sans atomes) invariante, alors $\F$ est transversalement hermitien. De plus, quitte à effectuer un revêtement ramifié au dessus d'hypersurfaces invariantes par $\F$, on peut choisir la métrique transverse à courbure constante.}

\vskip 10 pt

Un tel phénomène a été plus généralement conjecturé par \'Etienne Ghys pour les feuilletages {\it transversalement} holomorphes de codimension $1$ (\cite{ghys}) et prouvé dans le cas transversalement affine dans {\it loc.cit}. Comme le lecteur pourra le constater, la situation que nous regardons ici permet de "gommer" toutes les difficultés d'ordre dynamique posées par ce problème. Par exemple, il est facile de constater qu'un feuilletage transversalement affine sur une variété Kähler compacte $X$ est automatiquement transversalement euclidien. Ceci résulte du fait que tout fibré en droites plat $L$ sur $X$ admet une unique connexion plate hermitienne vis à vis de laquelle toute forme holomorphe à valeurs dans $L$ est fermée.


\section{Courants invariants par holonomie}
\subsection{Courants invariants}

Suivant \cite{sullivan}, il y a une correspondance naturelle entre mesures et courants positifs invariants par un  feuilletage. Typiquement, une mesure atomique peut se représenter en terme de courant d'intégration le long d'une feuille compacte. Essentiellement pour des raisons d'ordre cohomologique, il sera plus commode d'adopter le langage des courants invariants dont nous rappelons la définition (dans notre contexte).
\begin{definition}\label{courantinvariant}
Soit $\mathcal F$ un feuilletage holomorphe de codimension 1 (\'eventuellement singulier) sur une vari\'et\'e $X$ complexe et $T$ un courant positif ferm\'e de bidegr\'e $(1,1)$ d\'efini sur $X$. 
Par la suite, $X$ sera supposée K\"ahler compacte. On désigne par $\theta$ une forme de Kähler fixée une fois pour toutes.
 On dira que $T$ est $\mathcal F$-invariant (ou invariant par holonomie de $\mathcal F$) si au voisinage de tout point de $X$, on a
       $$T\wedge \omega=0$$
o\`u $\omega$ d\'esigne une 1 forme holomorphe d\'efinissant localement le feuilletage.
\end{definition}

Au voisinage d'un point ou le  feuilletage o\`u celui-ci est d\'ecrit par l'équation $\{dz=0\}$, un tel courant s'exprime donc sous la forme $T=ia(z)dz\wedge d\overline{z}$ o\`u $a$ est une mesure positive. De façon équivalente, $T=i\ddbar \varphi (z)$ où le potentiel local $\varphi $ peut être choisi comme fonction plurisousharmonique ({\it psh}) ne dépendant que de la seule variable $z$, ce qui revient à dire que $\varphi$ est constant sur les feuilles.

L'ensemble des courants positifs $\F$-invariants forme un cône qui sera noté $\CF$.

\begin{definition}

 Quand $T\in \CF$ est une forme lisse semi-positive qui ne s'annule pas, on dira que $T$ est (la forme volume d') une métrique hermitienne transverse au feuilletage.  Dans cette situation,  on peut écrire localement
 
 $$T=\frac{i}{\pi} e^{2\psi (z)}{dz}\wedge{\overline d{z}}$$
 
 où $z$ est une coordonnée transverse et lui associer sa forme de courbure 
 
 $$\rho_T=-\frac{i}{\pi} \ddbar \psi.$$
 
 On dira alors que $\F$ est transversalement 
 \begin{itemize}
 \item  {\bf euclidien} si $\rho_T=0$.
\item {\bf hyperbolique} si $\rho_T=-\eta_T$.
\item{\bf sphérique (ou elliptique)} si $\rho_T=\eta_T$.
\end{itemize}
\end{definition}

Chaque courant $T\in \CF$ représente une classe de cohomologie qui sera notée $\{T\}$.

\begin{definition}\label{defpsef}La classe de cohomologie $\alpha\in H^{1,1}(X,\RR)$ est dite {\bf pseudo-effective} si $\alpha$ peut \^etre repr\'esent\'ee par un courant positif ferm\'e de bidegr\'e $(1,1)$.

On dira alors qu'un fibr\'e  en droites holomorphe $L$ est {\bf pseudo-effectif} si $c_1(L)$ est pseudo-effective ou, de fa\c con \'equivalente,  si l'on peut implanter sur $L$ une m\'etrique $h(x,v)= {|v|}^2e^{-2\varphi (x)}$ o\`u le poids local $\varphi$ est une fonction plurisousharmonique. 

Le courant $T$ est alors \'egal \`a la forme de courbure d'une telle m\'etrique (\'eventuellement singuli\`ere) suivant la formule:
\begin{equation}\label{courbure}
 T=\frac{-i}{2\pi}\partial\overline{\partial}\log h=\frac{i}{\pi}\partial\overline{\partial}\varphi.
\end{equation}
\end{definition}

Nous ferons un usage répétitif du théorème de l'indice de Hodge:

\begin{thm} \label{Hodge}
Considérons  la forme bilinéaire symétrique $q$ définie sur $H^{1,1}(X,\mathbb R)$ par

  $$q(c,c')= c. c'.{\{\theta\}}^{n-2},$$


Soit $p=\mbox{dim}_{\mathbb R} H^{1,1}(X,\mathbb R)$. Alors la forme $q$ a $(1,p-1)$ comme signature.
\end{thm}

   Nous serons plus particulièrement amené à étudier les propriétés du fibré {\bf normal} du feuilletage  ${N}_{\mathcal
          F}={TX\over {\mathcal F}}$ ($\F$ étant vu ici comme sous-fibré de corang $1$ du fibré tangent $TX$) ou par dualité celles de son fibré {\bf conormal} $\con$.

   Considérons
un recouvrement ouvert $(U_i)_{i\in I}$ de $X$ suffisamment fin
tel que sur chaque ouvert de la famille, les feuilles de ${\mathcal
F}$ soient données par les niveaux d'une submersion holomorphe

                $$f_i:U_i\rightarrow{\mathbb C}.$$

Sur chaque intersection $U_i\cap U_j$, les différentielles des
$f_i$ sont liées par la relation

              $$df_i=g_{ij}df_j,\ g_{ij}\in {\mathcal O}^*(U_i\cap U_j)$$
où le cocycle multiplicatif $g_{ij}$ représente le fibré normal du
feuilletage $N_{\mathcal \F}$.

Le faisceau ${\mathcal \F}_{\infty}$ des germes de $(1,0)$ formes
différentielles lisses tangentes ${\mathcal \F}$ est fin, ce qui assure
l'existence de sections $\omega_i\in{\mathcal \F}_{\infty}(U_i)$
vérifiant sur $U_i\cap U_j$:

            $$\omega_i-\omega_j={dg_{ij}\over g_{ij}}.$$

Puisque les $\omega_i$ sont de la forme
                                           $$g_idf_i$$
où $g_i\in {\mathcal C}^\infty (U_i)$, on obtient par différentiation
que
\begin{equation}\label{chernclass}
       \eta=d\omega_i=dg_i\wedge df_i
\end{equation} 
   est bien définie  par recollement en tant que  $2$ forme fermée et représente à un facteur près la classe de Chern  $c_1(N_{\F})$.  On observe de plus que  $\eta\wedge\eta=0$. En particulier, on obtient (propriété d'annulation de Bott) que
   
 \begin{equation}\label{annulationdebott}
 c_1^2(N_\F)=0
 \end{equation}

\begin{lemme}\label{intersectionnulle}
Let $S,T\in \CF$, alors $\{S\}\{T\}=0$. En particulier, ${\{S\}}^2=0$. De plus les classes $\{S\}$ et $\{T\}$ sont colinéaires.
 


\begin{proof}

Soit  $(U_i)_{i\in I}$ un recouvrement ouvert de  $X$ choisi de telle sorte $T=\frac{i}{\pi}\partial\overline{\partial}\varphi_i$,  où $\varphi_i$ {\it psh} localement  constant sur les feuilles.. Sur les intersections $U_i\cap U_j$, on a  $\partial\varphi_i-\partial\varphi_j=\omega_{ij}$ où $\omega_{ij}$ est une $1$ forme holomorphe définie sur $U_i\cap U_j$ s'annulant sur les feuilles et qu'on peut donc identifier à une section locale du fibré conormal $N_\F^*$. Soit $\xi$ une $(1,1)$  forme lisse fermée telle que $\{\xi\}=\{T\}$. Quitte à prendre un recouvrement plus fin, on obtient donc que sur chaque $U_i$, $\xi=\frac{i}{\pi}\partial\overline{\partial}u_i$, $u_i\in{\mathcal C}^\infty (U_i)$ tel que $\partial u_i-\partial u_j=\omega_{ij}$. Par suite, la collection des $\frac{i}{\pi}\partial u_i\wedge S$, pour $S\in\CF$, produit par recollement une $(1,2)$ forme dont la différentielle est $\eta\wedge S$. Le premier point du  lemme est donc établi. La colinéarité est alors une conséquence immédiate du théorème \ref{Hodge}.
\end{proof}

\end{lemme}

\begin{lemme}\label{classesproportionnelles}
Soit $T\in\CF\setminus\{0\}$; alors il existe $\lambda\in \R$ tel que $c_1(\norm)=\lambda \{T\}$.

\end{lemme}

\begin{proof}
D'après (\ref{chernclass}), on a $T\wedge \eta=0$. Au niveau cohomologique, ceci implique que $\{T\}.c_1(\norm)=0$. On conclut en combinant le lemme \ref{intersectionnulle} et la propriété d'annulation  (\ref {annulationdebott}) avec le théorème de l'indice \ref{Hodge}.
\end{proof}





\begin{cor}
Soit $T\in \CF \setminus\{0\}$, alors $\norm$ ou $\con$ est pseudo-effectif.
\end{cor}

\section{Existence de métriques transverses invariantes}\label{métriquestransverses}

Dans la section précédente, nous avons vu que l'existence d'un courant positif invariant  par holonomie du feuilletage $\F$ implique certaines propriétés numériques de son fibré normal: celui-ci ou son dual sont pseudo-effectifs. Comme l'atteste ce qui suit, la situation $\con$ pseudo-effectif entraîne nécessairement l'existence d'une métrique invariante, sans supposer {\it a priori} la présence d'un courant positif invariant. 

{\bf Notations} - Nous désignerons par $PSH(U)$ l'ensemble des fonctions {\it psh} sur une variété complexe $U$
\vskip 5 pt

\subsection{Le cas $c_1(\norm)=0$}\label{c1=0}

\begin{thm}
Soit $\F$ un feuilletage holomorphe régulier de codimension un sur $X$ Kähler compacte. On suppose que $c_1( \norm)=0$; alors $\F$ admet une métrique euclidienne transverse invariante par holonomie.  
\end{thm}

\begin{proof}
Comme auparavant , on écrit que le feuilletage est défini par des submersions locales 

 $$f_i:U_i\rightarrow{\mathbb C}.$$

avec les conditions de recollement 

              $$df_i=g_{ij}df_j$$
ou  $(g_{ij})$ représente $N_{\mathcal \F}$. 

Dans le cas présent, on peut donc choisir des unités locales $u_i\in  {\mathcal O}^*(U_i)$ telles que 

$$|h_{ij}|=1$$

où $h_{ij}= {u_i}^{-1}u_j g_{ij}$ . On obtient donc que 

$$  \Omega_i=h_{ij}\Omega_j$$

où $\Omega_i=u_idf_i$ est définie sur $U_i$.  On hérite ainsi une $(1,2)$ forme globalement définie $\alpha$ qui vaut $\Omega_j\wedge d\overline{\Omega_j}$ en restriction à $U_j$. 
Soit $\theta$ une forme de Kähler sur $X$. Le théorème de Stockes nous donne 

$$  \int_X d\alpha \wedge \theta^{n-2}=0.$$
Puique  $-d\alpha=d\Omega_j\wedge d\overline{\Omega_j}$ est une $(2,2)$ forme positive, on conclut qu'elle est en fait nulle. Par suite $\eta=i\Omega_j\wedge d\overline {\Omega_j}$ est une métrique euclidienne $\F$-invariante.
\end{proof}

\subsection {Le cas $\con$ pseudo-effectif}\label{conpseudo}

\begin{thm}
Soit $\F$ un feuilletage holomorphe régulier sur $X$ Kähler compacte. On suppose que $\con$ est pseudo-effectif et  $c_1(\norm)\not=0$; alors $\F$ admet une métrique hyperbolique transverse invariante par holonomie.  
\end{thm}

\begin{proof}
Par définition, il existe  un courant positif fermé $T$ dont la classe de cohomologie $\{T\}$ est égale à $c_1(\con)$. A priori, il n'y a pas de raisons pour que $T\in \CF$. 

Soit $h$ une métrique sur $\con$ dont le poids local $\varphi$ est un potentiel local {\it psh} de $T$. 
Par dualité, on récupère donc une métrique sur $\norm$ une métrique de poids $-\varphi$ définie par la $(1,1)$ forme positive globale à coefficients  $L_{loc}^\infty$

$$\eta_T=\frac{i}{\pi}e^{2\varphi} \omega\wedge{\overline \omega}$$

\noindent où $\omega$ est une $1$ forme holomorphe non dégénérée définissant localement $\F$.

Puisque toute  fonction pluriharmonique sur $X$ est constante, le choix de $\eta_T$ (pour $T$ donné) est unique à multiplication près par une constante positive. Elle est donc canoniquement associé à $T$ si l'on fixe la normalisation 

$$\int_X\eta_T\wedge {\theta}^{n-1}=\int_X T\wedge{\theta}^{n-1}.$$


Ainsi que l'a observé Marco Brunella ( cf.\cite{brsurf}, lemme 10  ), la forme $\eta$ est en fait invariante par holonomie du feuilletage, ce qui revient à dire qu'elle est fermée (au sens des courants). Rappelons la démonstration de ce dernier point: 
Considérons le courant $\Omega=-\partial\overline{\partial}(\eta_T\wedge {\theta}^{n-2})$ où $n$ est la dimension complexe de $M$.
Au voisinage d'un point non situé dans l'ensemble singulier $\mbox {Sing}\ \mathcal F$ du feuilletage, on peut écrire en coordonnées holomorphes locales $\omega=fdz$ où $f$ est holomorphe inversible. Sur $U=M\setminus \mbox{Sing}\ \mathcal F $, on obtient donc que $$\Omega=- \partial\overline{\partial}(e^{2\varphi+2\log{|f|}})dz\wedge d\overline{z}\wedge{\theta}^{n-2}$$
 est positif (l'exponentielle d'une fonction {\it psh} est {\it psh}).
  On obtient alors que $\Omega$ est une mesure {\it positive} sur $X$ tout entier,  puis que $\Omega=0$ par exactitude. 
  
  On peut donc en déduire qu'au voisinage d'un point régulier, $e^{2\varphi+2\log{|f|}}$ est {\it pluriharmonique} dans les feuilles et finalement constante car l'exponentielle d'une fonction {\it psh} est pluriharmonique si et seulement si cette fonction est {\it constante}. Puisqu'on a montré que $\varphi +\log{|f|}$ ne dépend que de la variable $z$ (qui paramètre l'espace des feuilles) on peut donc conclure que $T$ et $\eta_T$ sont deux éléments de $\CF$. Posons $\CFN:=\{T\in \CF| \{T\}=c_1(\con)\}$ et considérons l'application 
  
 $$ \begin{array}{lrcl}
\beta : & \CFN & \longrightarrow & \CFN \\
    & T & \longmapsto & \eta_T \end{array} $$
     




Remarquons  que l'existence d'une métrique hyperbolique transverse est équivalente à l'existence d'un point fixe pour $\beta$. Sachant que $\CFN$ est un convexe  {\it compact} pour la topologie faible,  ce point fixe sera produit par le théorème de  Leray-Schauder-Tychonoff une fois que nous aurons montré que $\beta$ est  {\it continue}. C'est effectivement ainsi que l'on procède dans  \cite{to}(D\'emonstration du lemme 3.1, p.371)  dans un cadre un peu plus général puisqu'on y manipule également des feuilletages singuliers.  On rappelle quelle est l'idée de la démonstration: Soit  $(T_p)\in  
{\CFN}^{\mathbb N}$ une suite convergeante vers  $T$ telle que  $(\beta(T_p))$ est convergente. On doit vérifier que
$$\lim\limits_{p\to +\infty} \beta(T_p)=\beta(T).$$
 Sur un ouvert $U$ de $X$ isomorphe à un polydisque, on peut écrire que  $T_p=\frac{i}{\pi}\ddbar\varphi_p^U$ où chaque $\varphi_p^U\in PSH(U)$. Soit $\mathcal U$ un recouvrement de $X$ par un nombre fini de tels ouverts. Sur chaque ouvert $U$ de $\mathcal U$, le feuilletage est définie par une forme intégrable $\omega_U$ ne s'annulant pas. Ces formes se recollent sur les intersections via le cocycle $g_{UV}\in {\mathcal O}_ {U\cap V}^*$:
   
   $$\omega_U=g_{UV}\omega_V$$
   
   qui représente le fibré normal $\norm$. Pour chaque $p$, on peut de plus choisir les $\varphi_{p}^U,U\in\mathcal U$ de telle sorte que
   
   $$ \mbox{Max}_{U\in\mathcal U}\ \mbox{Sup}_{x\in U}\varphi_{p}^U(x)=0\ ,\ \varphi_{p}^{U}-\varphi_{p}^{V}=-\log |g_{UV}|$$
   
   On obtient  donc que $$ \beta(T_p)=e^{2\lambda_p}\frac{i}{\pi}e^{2\varphi_p^U}{\omega^U\wedge\overline{\omega^U}}$$   
   
   où $\lambda_p$ est une constante réelle.

 En utilisant  \cite{ho} (théorème 4.1.9), on peut alors supposer, quitte à extraire une sous-suite, que  $(\varphi_p^U)$ est uniformément majorée sur les  compacts de $U$ et  converge in $L_{loc}^1 (U)$ vers  une   fonction $\varphi^U\in PSH(U)$ qui satisfait nécessairement

$$\frac{i}{\pi}\ddbar\varphi^U=T.$$

Par le théorème de convergence dominée, on obtient que 

$$e^{2\varphi_p^U}{\omega^U\wedge\overline{\omega^U}}$$ converge dans $L_{ loc}^1(U)$ vers 

$$e^{2\varphi^U}{\omega^U\wedge\overline{\omega^U}}.$$ 
 En particulier, la suite $(\lambda_p)$ est bornée; on peut donc la supposer convergente vers $\lambda\in\R$. 
 
  On peut alors conclure de même que  
 
 $$e^{2\lambda_p}\frac{i}{\pi}e^{2\varphi_p^U}{\omega^U\wedge\overline{\omega^U}}$$ converge dans $L_{ loc}^1(U)$ vers 
 
 $$e^{2\lambda}\frac{i}{\pi}e^{2\varphi^U}{\omega^U\wedge\overline{\omega^U}}$$
 Ce dernier terme est donc un élément de $\CFN$ qui n'est rien d'autre que $\beta(T).$

On a donc $$\lim\limits_{n\to +\infty}\beta(T_p)=\beta(T).$$

\end{proof}

\subsection{Le cas $\norm$ pseudo-effectif}\label{npsef}
   \begin{definition}
   Soit $T\in \CF$. On dira que $T$ est résiduel si tous ces nombres de Lelong sont nuls en codimension $1$. 
   \end{definition}
   
 \begin{rem}
 Puisqu'on peut choisir les potentiels locaux de $T$ constants sur les feuilles (et donc ne dépendant que d'une variable), on constate facilement que $T$ est résiduel si et seulement si tous ces nombres de Lelong sont nuls.
 
 \end{rem}
 
 \begin{rem}
 Soit  $\mathcal P$ l'ensemble des diviseurs premiers de $X$. Soit $T\in \CF$; on note $\lambda_D\geq 0$ le nombre de Lelong de $T$ sur $D\in \mathcal P$ et $[D]$ le courant d'intégration sur $D$ . Notons que $[D]$ est $\F$-invariant (et en particulier lisse) dès que $\lambda_D$ et que de plus $\F$ est une fibration holomorphe à fibres connexe si le nombre de tels diviseurs $D$ est infini (cf.\cite{ghysjoua}).

 Remarquons également que la décomposition de Siu de $T$ (valable plus généralement pour tout courant positif fermé) s'écrit dans notre cas
 
 $$T= \sum_{D\in\mathcal P} \lambda_D [D] + R$$
 
 où $R\in\CF$ esr résiduel.
  \end{rem}
  

Compte-tenu du lemme \ref{classesproportionnelles}, de \ref{conpseudo} et \ref{c1=0}, il reste à considérer le cas ou $\norm$ est pseudo-effectif et non numériquement trivial ($c_1(\norm)\not=0)$.

\begin{prop}
Soit $\F$ un feuilletage holomorphe régulier sur $X$ Kähler compacte. On suppose que  $\CF\not=0$  mais que $\F$ n'admet pas d'hypersurface invariante; alors $\F$ admet une métrique hermitienne transverse de courbure constante invariante par holonomie . Il s'agit plus précisément d'une métrique sphérique  lorsque $\con$ n'est pas pseudo-effectif.
\end{prop}

\begin{proof}

Par définition, il existe  un courant positif fermé $T\in \CF$ dont la classe de cohomologie $\{T\}$ est égale à $c_1(\norm)$. Quitte à multiplier par un réel positif, on peut supposer que $\{T\}=c_1(\norm)$.  

Comme précédemment, on récupère donc  
 $h$ une métrique sur $\norm$ dont le poids local $\varphi$ est un potentiel local {\it psh} de $T$ et peut être choisi de plus localement constant sur les feuilles.
 
 \begin{rem}\label{intégrable}
 Puisque $T$ ne dépend que d'une seule variable, il ne peut admettre de nombre de Lelong strictement positif en l'absence d'hypersurface invariante. dans cette situation $e^{-\lambda\varphi}$ est localement intégrable pour toute constante $\lambda\in\R$ (cf. par exemple \cite{de1}).
\end{rem}


Cette métrique $h$ est donc incarnée par la $(1,1)$ forme positive à coefficients $L_{loc}^1$ (unique modulo constante multiplicative)

$$\eta_T=\frac{i}{\pi}e^{-2\varphi} \omega\wedge{\overline \omega}$$

\noindent où $\omega$ est une $1$ forme holomorphe non dégénérée définissant localement $\F$.

Soit $z$ une coordonnée locale  transverse au feuilletage dans laquelle  dans laquelle $\omega=gdz$ où $g$ est une unité.  On peut donc écrire 

   $$\eta_T=\frac{i}{\pi}e^{-2\varphi} {|g|}^2 dz \wedge d{\overline z}$$

   On peut estimer le $\ddbar$ de $\eta_T$ comme courant. En gardant à l'esprit que $\varphi$ est constant sur les feuilles , on obtient que 
   
   $$i\ddbar \eta_T= - \frac{1}{\pi}e^{-2\varphi} dg\wedge d{\overline g}\wedge dz\wedge d{\overline z}.$$
   
   Il s'agit d'un $(2,2)$ courant positif  $\ddbar$ exact et donc nul puisque la variété ambiante est Kähler compacte. Il s'ensuit que $g$ est également constant sur les feuilles et que finalement $\eta_T$ est fermé et appartient à $\CF$. Comme précédemment, définissons $\CFN:=\{T\in \CF| \{T\}=c_1(\norm\})$ et remarquons que $\eta_T$ est unique dès lors qu'on s'est fixé la normalisation $\eta_T\in \CFN$. Là encore, le résultat sera établi dès que l'application
   
   $$ \begin{array}{lrcl}
\beta : & \CFN & \longrightarrow & \CFN \\
    & T & \longmapsto & \eta_T \end{array} $$
    
    admet un point fixe et l'existence de ce dernier est assuré si on réussi à montrer que $\beta$ est continue (pour la topologie faible).
    
    On procède de façon similaire à la preuve du théorème \ref{conpseudo}. En adoptant les mêmes notations, on obtient ainsi sur chaque ouvert $U$ du recouvrement une suite de fonctions $(\varphi_p^U)\in {PSH (U)}^{\mathbb N}$ telle que 
    
    \begin{enumerate}[-]
    \item $ \beta(T_p)=e^{2\lambda_p}\frac{i}{\pi}e^{2\varphi_p^U}{\omega^U\wedge\overline{\omega^U}}$   
   où $\lambda_p$ est une constante réelle.
   \item $(\varphi_p)$ converge in $L_{loc}^1 (U)$ vers  une   fonction $\varphi^U\in PSH(U)$ qui satisfait nécessairement  $\frac{i}{\pi}\ddbar\varphi^U=T.$    
    \end{enumerate}


Pour conclure, il suffit encore d'établir que

$$e^{2\varphi_p^U}{\omega^U\wedge\overline{\omega^U}}$$ converge dans $L_{ loc}^1(U)$ vers 

$$e^{2\varphi^U}{\omega^U\wedge\overline{\omega^U}}.$$ 

Ce dernier point résulte de la remarque \ref{intégrable} et de \cite{deko} (main theorem 0.2, item (2)). 
\end{proof}

\subsection{Fin de la preuve du théorème principal}
 
 Comme on l'a vu, le fait que $\con$ soit pseudo-effectif entraîne automatiquement que $\F$ est transversalement hyperbolique. Dans ce cas, le théorème principal de l'introduction est donc établi.
 
 Quand  $\norm$ est pseudo-effectif (l'autre cas à envisager en présence de courants positifs invariants), l'existence de métrique invariante n'est pas du tout assurée. Un exemple typique est un  feuilletage construit par suspension sur un ${\mathbb P}^1$ fibré $S$ au dessus d'une courbe projective $\mathcal C$ de genre $g\geq 2$ à partir de la donnée d'une représentation
 
    $$\rho: \pi_1 (\mathcal C) \rightarrow \mbox {Aut}\ {\mathbb P}^1=PSL(2,\mathbb C).$$
    
    Par transversalité à la fibre $F={\mathbb P}^1$, on a  $\norm . F=2$ et de plus ${\norm}^2=0$ (annulation de Bott). D'après \cite{friedman}, proposition 15 p.124,  ceci garantit que $\norm$ est pseudo-effectif.
    
  L'existence d'un courant positif invariant par $\F$ est ici équivalente à l'existence d'un courant positif sur ${\mathbb P}^1$ invariant sous l'action de $G=\rho(\pi_1(V)$. En choisissant convenablement  $\rho$, on peut facilement produire des exemples qui n'en possédent aucun. En choisissant $\rho$ de telle sorte que $G$ fixe 1 ou 2 points de ${\mathbb P}^1$, on aboutit a des situations le feuilletage comporte une ou deux courbes invariantes dont les courants d'intégration sont les seuls invariants par holonomie.
 
 On peut néanmoins espérer l'existence d'une telle métrique en présence d'un courant résiduel non trivial $R\in\CF$.
 
 C'est l'objet de l'énoncé qui suit
 
 \begin{thm}
 Soit $\F$ un feuilletage holomorphe régulier de codimension 1 sur $X$ Kähler compacte. On suppose de plus que $\con$ n'est pas pseudo-effectif et qu'il existe un courant résiduel non trivial $R\in \CF$. Alors $\F$ est transversalement hermitien, de plus, quitte à passer à un revêtement fini ramifié au dessus d'hypersurfaces invariants par $\F$ (qui n'affecte pas la régularité de $X$ et du feuilletage), on peut supposer  que cette métrique est sphérique.
 \end{thm}

 \begin{proof}
 D'après l'étude menée précédemment, le seul cas à considérer reste $\norm$ pseudo-effectif non numériquement trivial avec présence d'au moins une hypersurface invariante.

Examinons d'abord le cas ou toutes les feuilles de $\F$ sont fermées (par \cite{ghysjoua}, c'est par exemple le cas dès qu'il y en a une infinité) et où le feuilletage définit donc une fibration 

$$\rho:X\rightarrow S$$

 à fibres connexes au dessus d'une surface de Riemann $S$.  Soient $F_1,...F_p$ les fibres multiples de $\rho$ de multiplicités respectives $n_1,...,n_p$ avec $2\leq n_1\leq n_2\leq....\leq n_p$. Ceci induit sur $S$ une structure d'orbifold définie par les points multiples $(s_1,n_1),...,(s_p,n_p)$, $s_i=\rho (\mbox{Supp}\ F_i)$. Par partition de l'unité, on peut implanter sur $S$ une métrique orbifold $\nu$ compatible aux données orbifold précédentes. En particulier $\rho^*\nu$ est une métrique transverse à $\F$ invariante  par holonomie.
 
  L'hypothèse sur $\norm$ assure que $S\simeq{\mathbb P}^1$ et que sa caractéristique d'Euler orbifold
 
 $$\chi(S_{orb})=2 - \sum_i (1-\frac{1}{n_i })$$ est strictement positive. En effet, si $\chi_{orb}(S) \leq 0$, $S$ serait uniformisable (en tant qu'orbifold) et admettrait à ce titre une mérique orbifold euclidienne ou hyperbolique $h$. Par suite, $\rho^* h$ serait une métrique lisse sur $\norm$ à courbure $\leq 0$, ce qui est incompatible avec la positivité de $\norm$.
 
 On conclut donc que $\F$ admet une métrique hermitienne invariante de courbure $1$ dès que $S_{orb}$ est uniformisable. 
 
 Il reste à traiter le cas des mauvaises orbifold, lesquelles sont associées aux deux situations suivantes:
 
 \begin{enumerate}
 \item $p=1$ \label{p=1},
 
 \item $p=2$ avec $2\leq n_1<n_2$\label{p=2}.

\end{enumerate}

Dans le cas (\ref{p=1}), considérons une fibre $F$ de $\rho$ autre que $F_1$. Soient $f_1, f$ des sections respectives de $\mathcal O ({F_1}_{{}_{red}})$ et ${\mathcal O}(F)$ s'annulant sur ${F_1}_{{}_{red}}$ et $F$. En posant

$$z=\frac{{f}^{1/n_1}}{f_1}$$

on récupère sur $X$ la $(1,1)$ forme (à pôle sur $F$)

$$\eta=\frac{i}{\pi}\frac{dz\wedge d\overline{z}}{{(1+{|z|}^2)}^2}$$

  Il s'agit d'un courant  positif invariant par $\F$ qui peut s'interpréter comme une métrique "orbifold" sur $X$ de courbure positive constante.
  
  Soit 
  
  $$r: \pi_1(X\setminus F)\rightarrow PSU(2,\mathbb C)$$
  
  la représentation définie par la monodromie de l'intégrale première multivaluée $z$.
  
 On peut naturellement y associer un revêtement  galoisien fini $g:\tilde X\rightarrow X$ ramifiant exactement dessus de $F$ à l'ordre $n_1$. 

Par construction, $g^* \eta$ définit une métrique hermitienne invariante de $g^*\F$.

Le cas (\ref {p=2}) se traite similairement en notant cette fois 

  $$z=\frac{{f_2}^{1/n_1}}{{f_1}^{1/n_2}}$$  où les $f_i$ sont des sections de  ${\mathcal O} ({F_i}_{{}_{red}})$ s'annulant sur  ${F_i}_{{}_{red}}$. Le revêtement considéré ici ramifie au dessus de $\mbox{ Supp}\ F_1$ et $\mbox{Supp}\ F_2$ aux ordres respectifs $n_1$ et $n_2$.
  \vskip 5 pt
  
  On suppose dorénavant que $\F$ nest pas une fibration.

 On reprend les notations de la section \ref{npsef} et l'on introduit le sous-ensemble $\CFR\subset \CFN$ formé des courants résiduels de $\CFN$.  Remarquons que l'application $\beta$ considérée dans ladite section est bien définie sur $\CFR$ et que $\beta (\CFR)\subset \CFR$. Soit $T\in \CFR$. Par les propriétés de régularités classiques de l'équation de Poisson $\Delta u=f$  (le fait  que $u$ gagne deux crans supplémentaires de différentiabilité par rapport à $f$), on constate que $\beta (T)$ définit une métrique hermitienne transverse $\F$ invariante à coefficients {\it continus} (et même de classe ${\mathcal C}^{2-\varepsilon}$).  
 
 Soit $H$ une hypersurface invariante et $\rho: \pi_1(m, X)\rightarrow \mbox{Diff} (\C, 0)$ sa représentation d'holonomie, cette dernière étant évaluée sur un germe de  tranversale $(T,m)$ au feuilletage .  Soit $G$ l'image de cette représentation.
 
  L'existence de $\beta (T)$ entraîne  facilement que $G$ est analytiquement linéarisable et plus précisément conjugué à un sous-groupe du groupe des rotations $R= \{h_\theta (z)=e^{2i\pi \theta} z, \theta\in\mathbb R\}$. 
  
En particulier, ${|z|}^2$ s'étend par holonomie en une fonction $f$ définie au voisinage de $H$,  constante sur les feuilles  et valant $0$ sur $H$. Soit $\varepsilon>0$ et  $\phi:\mathbb R\rightarrow[0,+\infty[$ une fonction lisse, $\phi (1)=1$ et dont le support est  contenu dans  $[1-\varepsilon,1+\varepsilon]$. Posons $\Psi= \phi\circ e^ f$. Pour $\varepsilon$ suffisamment petit, $\Psi$ a un sens et définit une fonction dans  ${\mathcal C}^\infty (X)$ constante sur les feuilles. 

De même , $T\in \CF$  où $T$ est la $(1,1)$ forme lisse définie par 

$$ T=i\Psi dz\wedge d\overline{z}.$$

On peut évidemment supposer que $T\in \CFN$.

On obtient donc que  $\beta (T)$ est  une métrique hermitienne transverse invariant par holonomie. Le feuilletage réel (de codimension $2$) sous-jascent à $\F$ est en particulier transversalement Riemannien; sachant que $\F$ n'est pas une fibration et admet au moins une hypersurface invariante, on conclut, suivant \cite{mo} , que $\F$ possède au plus $2$ feuilles fermées dont nous noterons $H$ la réunion, les autres feuilles s'accumulant sur des hypersurfaces analytiques réelles compactes de $X\setminus H$.
Supposons dans un premier temps  que $H=F_1\cup F_2$ est {\it exactement} une union de 2 feuilles fermées.  Par le théorème de l'indice de Hodge, il existe des entiers positifs $m,n$, qu'on peut supposer premiers entre eux, tels que numériquement on ait  $n_1F_1= n_2F_2$. On procède alors comme précédemment en posant  $$z=\frac{{f_2}^{1/n_1}}{{f_1}^{1/n_2}}$$  

Il s'agit d'une fonction multivaluée dont le module est en revanche bien défini. Par le principe du maximum et compte-tenu de la dynamique de  $\F$ décrite ci-dessus, $|z|$ et donc $z$ sont constants sur les feuilles. 

Comme avant, on hérite de la métrique orbifold  invariante à courbure positive constante

$$\eta=\frac{i}{\pi}\frac{dz\wedge d\overline{z}}{{(1+{|z|}^2)}^2}$$

qui est par ailleurs une vraie métrique dès lors que $n_1=n_2=1$.

On supposera par la suite $n_1\not=n_2$.

Soit $G$ l'image de la représentation 

$$r: \pi_1(X\setminus( F_1\cup F_2))\rightarrow PSU(2,\mathbb C)$$ 

associée à la monodromie de $z$.

Par le théorème de Malcev, il existe dans $G=\mbox{Im}\ r$ un sous groupe normal $H$ d'indice fini sans torsion. 
 Ceci assure à nouveau l'existence d'un revêtement Galoisien fini  $g:\tilde X\rightarrow X$ ramifiant exactement au dessus de $F_1$ et $F_2$ aux ordres respectifs  $n_2$ et $n_1$. La métrique lisse $g^*\eta$ satisfait donc les propriétés requises.
 \vskip 10 pt
 
 Le cas où $H$ est réduit à une seule feuille se ramène facilement au précédent. En effet, d'après \cite{mo}, le pseudo-groupe d'holonomie $G$ du feuilletage est un pseudo-groupe de Lie d'isométries dont l'algèbre de Lie $\mathfrak{G} $ est décrite  un faisceau localement  de champs basiques de Killing. Compte-tenu du fait que $\F$ est transversalement holomorphe et de sa dynamique, $\mathfrak {G}$ est de dimension un réelle et est localement engendré par la partie réelle d'un champ de vecteurs holomorphe $\mathcal V$ qui dans une coordonnée locale appropriée s'écrit 
 
 \begin{enumerate}
 \item $\mathcal V=\frac{\partial }{\partial z}$ en dehors de $H$
 \item $\mathcal V=iz\frac{\partial }{\partial z}$ au voisinage de $H=\{z=0\}$.
 \end{enumerate} 
 
 Soit $\rho:\pi_1(X)\rightarrow {\mathbb R}^*$ la représentation de monodromie associée à $\mathfrak G$ (vu comme système local de rang $1$) et soit $E$ le fibré en droite plat sous-jascent. Par dualité, le feuilletage est définie par une forme logarithmique  $\omega$ à valeurs dans $E^*$ dont  le diviseur des  pôles est exactement $H$ et telle que $\nabla_{{\rho}^*}\omega=0$ où $\nabla_{{\rho}^*}$ est la connexion plate associée à la représentation duale $\rho^*$. En particulier, $\rho$ ne peut être triviale en vertu du théorème des résidus. Comme $X$ est Kähler, $E^*$ admet une unique connexion plate unitaire $\nabla_u$ et on peut vérifier qu'on a encore $\nabla_u (\omega)=0$. Pour montrer cela, on peut utiliser le même type d'argument que celui développé dans \cite{brmen} (correspondant à $E$ trivial):  Puisque le résidu de $\omega$ le long de $H$ est une section de $E^*$ (en particulier ce dernier est trivial en restriction à $H$), on constate que $\nabla_u\omega_0$ est une {\it 2 forme  holomorphe} à valeur dans $E^*$. Puisque $E^*$ est unitaire, la quantité $\Omega=\nabla_u\omega\wedge \overline{\omega}$ est bien défini en tant que $(2,1)$ forme intégrable. Par la règle de Leibnitz, la composante de bidegré $(2,2)$ de sa différentielle $d\Omega$ (au sens des courants) vaut
 
 $$ {(d\Omega)}^{2,2}= \nabla_u\omega\wedge\overline{\nabla_u\omega }$$
 
 Par ailleurs, en choisissant une forme de Kähler $\theta$ sur $X$, on obtient par la formule de Stokes que 
 
 $$\int {(d\Omega)}^{2,2}\wedge\theta^{n-2}=\int d(\Omega\wedge \theta^{n-2})=0 .$$
 
 Par positivité, on conclut bien que   $\nabla_u\omega=0$.   
 
 Ainsi, la différence $\nabla_{{\rho}^*}-\nabla_u$ est une forme holomorphe (et donc fermée) définissant le feuilletage. Puisque $H$ est la seule feuille fermée, cette forme est nécessairement triviale, ce qui veut dire que $\rho$ est une représentation unitaire à valeurs dans $\{-1,1\}$. Il est par ailleurs facile de constater que la représentation induite sur un petit voisinage de $H$ est triviale. Le feuilletage induit sur le revêtement double associé à $\rho$ admet par conséquent deux hypersurfaces invariante et on conclut comme précédemment.
 
 \end{proof}
 
 \begin{rem}
 
 Comme ci-dessus, on pourrait utiliser l'existence d'un faisceau localement constant de champs de Killing transverses décrivant la dynamique d'un feuilletage riemannien pour conclure en toute généralité à l'existence d'une métrique transverse à courbure constante pour $\F$. Ceci permet de s'affranchir de la méthode de point fixe décrite dans la section \ref{métriquestransverses}. 
 
Dans ce cas,  pour établir que $\F$ est riemannien sous la seule hypothèse de la donnée d'un courant invariant diffus $T$, on montre d'abord, comme on l'a fait en section \ref{métriquestransverses} qu'à normalisation près, $T$ est la forme de courbure d'une métrique hermitienne transverse invariante à coefficients continus (par régularité du laplacien). Par \cite{asu}, Corollary 4.23, ceci suffit à prouver l'existence d'une métrique invariante {\it lisse}.
 
 \end{rem}
 
 \section{Une remarque à propos des feuilletages singuliers}

 \subsection{Feuilletages à singularités hyperboliques}\label{hyperboliques}
 On s'intéresse au problème suivant. Soit $S$ une surface kählérienne munie d'un feuilletage holomorphe $\F$. On suppose de plus que le lieu singulier $\mbox{Sing}\ \F$ de $\F$ (un ensemble fini de points) est non vide et que toutes les singularités sont hyperboliques, c'est-à-dire définies localement par des formes du type $\omega=\lambda xdy+ydx$ avec $\lambda\in{\mathbb C}\setminus\mathbb R$.
 
 Suivant \cite{brcourbesentières}, tout courant  fermé positif diffus $T$ $\F$ invariant est supporté sur le complémentaire de $ \mbox{Sing}\ \F$. En reprenant la démonstration du lemme \ref{intersectionnulle}, on obtient que

 \begin{equation} \label{autointersection}
 {\{T\}}^2=0
 \end{equation}
 
 \subsection{Non existence de courants diffus invariants}
 
 Soit $\F$ un feuilletage satisfaisant les hypothèses ci-dessus. Nous ne connaissons pas d'exemples possédant un courant positif fermé diffus invariant (ou de façon équivalente de mesures positives diffuses invariantes).  C'est un fait bien connu sur le plan projectif ${\mathbb P}^2$. En effet, si $c\in H^{1,1}({\mathbb P}^2, \mathbb R) $ est non triviale, $c^2$ est non triviale, ce qui contredit (\ref{autointersection}) si $c=\{T\}$.\footnote{Plus généralement, Brunella a montré dans \cite{brinexistence} que cette propriété de non existence persistait pour les feuilletages en courbes à singularités hyperboliques dans ${\mathbb P}^n$, ce qui requiert des arguments bien plus subtils dès que $n>2$.} 
 
 On peut légèrement améliorer ce résultat en l'étendant au cadre suivant  
 
 \begin{thm}
Soit $\F$ un feuilletage holomorphe en courbes satisfaisant les hypothèses de \ref{hyperboliques}. On suppose de plus que ${N\F}^2\geq 0$. Alors $\F$ n'admet pas de mesure diffuse transverse invariante.
 
 \end{thm}
 
 \begin{proof}
 Supposons par l'absurde qu'une telle mesure (vu comme courant)  $T$ existe. Puisque $T$ est supporté dans $X\setminus \mbox{Sing}\ \F$, on a (par exemple en reprenant la démonstration du lemme \ref{classesproportionnelles} )
 \begin{equation}\label{intersectioncourantnormal}
 \{T\}.c_1(N\F)=0.
 \end{equation}
  En combinant ceci  avec la propriété d'annulation \ref{autointersection} et ${N\F}^2\geq 0$, on conclut par le théorème de l'indice de Hodge que $c_1(N\F)=\lambda \{T\}$ avec $\lambda\in\mathbb R$.  On récupère donc (voir section \ref{métriquestransverses}) sur $N\F$ une métrique représentée par la $(1,1)$ forme singulière (globale)
  
  $$\eta=\frac{i}{\pi}e^{-2\lambda\varphi} \omega\wedge{\overline \omega}$$
  
  où $\varphi$ est un potentiel local de $T$ et $\omega$ une $1$ forme holomorphe engendrant localement $\F$.   En dehors de $\mbox{Sing}\ \F$, on peu choisir $\varphi $ constant sur les feuilles. Puisque $T$ est diffus et  $\eta$ est lisse au voisinage de $\mbox{Sing}\ \F$ (car $\varphi$ y est pluriharmonique), $\eta$ est à coefficients $L_{loc}^1$ et donc bien définie en tant que courant. Ceci permet encore de conclure que $ i\partial\overline{\partial}\eta$ est  positif sur $X\setminus\mbox{Sing}\ \F$ et donc sur $X$ tout entier par lissité aux singularités. Par Stokes, on en déduit de même que $\eta$ est un courant positif {\it fermé} invariant par $\F$; par ailleurs ce courant est diffus et de support $X$, ceci produit la contradiction recherchée.

 \end{proof}




\bigskip





 

\begin{thebibliography}{99}
\bibitem{asu} T.ASUKE, {\it A Fatou-Julia decomposition of transversely holomorphic foliations}, Tome 60, n° \textbf{3} (2010), p1057-1104.
\bibitem{brbir} M.BRUNELLA, {\it Birational geometry of foliations}, Publicações Matemáticas do IMPA, [IMPA Mathematical Publications] Instituto de Matemática Pura e Aplicada (IMPA), Rio de Janeiro, 2004.
\bibitem{brcourbesentières}  M.BRUNELLA, {\it Courbes entières et feuilletages holomorphes}, L'enseign. Math. \textbf{45}, 1999, 195-216.
 \bibitem{brsurf} M.BRUNELLA, {\it Feuilletages holomorphes sur les surfaces complexes compactes}, Ann. Sci. École Norm. Sup. (4)  {\bf 30}  (1997),  no. 5, 569-594. 
 \bibitem{brinexistence} M.BRUNELLA, {\it Inexistence of invariant measures for generic rational differential equations in the complex domain}, Bol. Soc. Mat .Mexicana (3) Vol. \textbf{12}, 2006.
 \bibitem{brmen} M.BRUNELLA, L.G.MENDES, {\it Bounding the degree of solutions to Pfaff equations}, Publications Matemàtiques, Vol.\textbf {44} (2000), 593-604.

149-156.

\bibitem{de} J.P DEMAILLY, {\it On the Frobenius integrability of certain holomorphic $p$-forms},  Complex geometry (Göttingen, 2000),  93-98, Springer, Berlin, 2002.
\bibitem{de1} J.P DEMAILLY, {\it $L^2$ vanishing theorems for positive line bundles and adjunction theory}, Transcendental methods in algebraic geometry.
Lectures given at the 3rd C.I.M.E. Session held in Cetraro, July 4-12, 1994. Edited by F. Catanese and C. Ciliberto. Lecture Notes in Mathematics, {\bf 1646} (1-97).

\bibitem{deko} J.P DEMAILLY, J.KOLL\'AR,
{\it Semi-continuity of complex singularity exponents and Kähler-Einstein metrics on Fano orbifolds}, Annales scientifiques de l'École Normale Supérieure, Sér. 4, {\bf 34} no. 4 (2001), p. 525-556 
\bibitem{friedman} R. FRIEDMAN, {\it Algebraic Surfaces and Holomrphic Vector Bundles}, Springer Verlag (1998).
\bibitem{ghysjoua} \'E. GHYS,{\it \`A propos d'un théorème de J.-P. Jouanolou concernant les feuilles fermées des feuilletages holomorphes}, Rendiconti del Circolo Matemetico di Palermo, Serie II, Tomo XLIX (2000), pp.175-180.
\bibitem{ghys} \'E. GHYS, {\it Flots transversalement affines et tissus feuilletés}, Mémoires de la S.M.F. 2$^e$ série, tome \textbf{46} (1991), p.123-150.
\bibitem{ho} L.HÖRMANDER,  {\it The analysis of linear partial differential operators. I. Distribution theory and Fourier analysis},  Classics in Mathematics. Springer-Verlag, Berlin, 2003. 
\bibitem{mo} P.MOLINO, {\it Riemannian foliations}, Translated from the French by Grant Cairns. With appendices by Cairns, Y. Carri\`ere, \'E. Ghys, E. Salem and V. Sergiescu. Progress in Mathematics, {\bf 73.} Birkh\"auser Boston, Inc., Boston, MA, 1988, 339 pp.
\bibitem{sullivan} D.SULLIVAN, {\it Cycles for the Dynamical Study of Foliated Manifolds and Complex Manifolds}, Inventiones math. \textbf{36}, 225-255 (1976).
\bibitem{to} F.TOUZET,{\it Uniformisation de l'espace des feuilles de certains feuilletages de codimension un}, Bull Braz Math Soc, New Series {\bf 44}(3) (2013), 351-391.
\end{thebibliography}
\end{document}